\documentclass[a4paper, 11pt]{article}
\usepackage[preprint]{optional}
\usepackage[english]{babel}
\usepackage[T1]{fontenc}
\usepackage{amsmath}
\usepackage{amsthm}
\usepackage{amsfonts}
\usepackage{amssymb}
\usepackage{graphicx}
\usepackage{paralist}
\usepackage{tikz}
\usepackage{url}
\usepackage{pdfsync}
\usepackage{bm}

\opt{mydraft}{
\usepackage{refcheck}
\usepackage{todonotes}
}
\providecommand{\todo}[2][]{}

\theoremstyle{plain}
\newtheorem{thm}{Theorem}[section]
\newtheorem{prop}[thm]{Proposition}
\newtheorem*{prop*}{Proposition}
\newtheorem{lem}[thm]{Lemma}
\newtheorem*{lem*}{Lemma}

\newtheorem*{cor*}{Corollary}

\theoremstyle{definition}
\newtheorem{defn}[thm]{Definition}
\newtheorem*{defn*}{Definition}
\newtheorem{rem}[thm]{Remark}
\newtheorem*{rem*}{Remark}
\newtheorem{example}[thm]{Example}
\newtheorem*{example*}{Example}

\DeclareMathOperator{\im}{Im}
\DeclareMathOperator{\ind}{ind}

\DeclareMathOperator{\re}{Re}
\DeclareMathOperator{\sign}{sign}

\newcommand{\abs}[1]{\vert #1 \vert}
\newcommand{\abslr}[1]{\left\vert #1 \right\vert}

\newcommand{\R}{\ensuremath{\mathbb{R}}}
\newcommand{\C}{\ensuremath{\mathbb{C}}}
\newcommand{\Z}{\ensuremath{\mathbb{Z}}}

\newcommand{\bigO}{\ensuremath{\mathcal{O}}}

\newcommand{\defby}{\mathrel{\mathop:}=}

\newcommand{\coloneq}{\mathrel{\mathop:}=}

\newcommand{\conj}[1]{\overline{#1}}

\newcommand{\titlestr}{The index of singular zeros of harmonic mappings of
anti-analytic degree one}
\newcommand{\authrob}{Robert Luce}
\newcommand{\autholi}{Olivier S\`{e}te}

\newcommand{\affilrob}{Ecole polytechnique f\'ed\'erale de Lausanne, Switzerland}
\newcommand{\affiloli}{University of Oxford, United Kingdom}

\newcommand{\emailrob}{robert.luce@epfl.ch}
\newcommand{\emailoli}{olivier.sete@maths.ox.ac.uk}

\newcommand{\urlrob}{http://people.epfl.ch/robert.luce}
\newcommand{\urloli}{http://www.maths.ox.ac.uk/people/olivier.sete}

\newcommand{\ackstr}{This research was supported through the programme
``Research in Pairs'' by the Mathematisches Forschungsinstitut
Oberwolfach in 2017.}

\begin{document}

\opt{preprint}{%
    \title{\titlestr\thanks{\ackstr}}
    \author{\authrob\footnote{\affilrob, {\ttfamily\emailrob,
    \urlrob}} \and \autholi\footnote{\affiloli, {\ttfamily\emailoli,
    \urloli}}}
}
\opt{tandf}{%
    \title{\titlestr\thanks{\ackstr}}
    \author{
    \name{\authrob\textsuperscript{a}$^\dagger$\thanks{$^\dagger${\ttfamily\emailrob, \urlrob}}
    and
    \autholi\textsuperscript{b}$^\ddagger$\thanks{$^\ddagger${\ttfamily\emailoli,
    \urloli}}}
    \affil{\textsuperscript{a}\affilrob\\ \textsuperscript{b}\affiloli}
}
}
\maketitle

\begin{abstract}
We study harmonic mappings of the form $f(z) = h(z) - \overline{z}$,
where $h$ is an analytic function.   In particular we are interested
in the index (a generalized multiplicity) of the zeros of such
functions.  Outside the critical set of $f$, where the Jacobian of $f$
is non-vanishing, it is known that this index has similar properties
as the classical multiplicity of zeros of analytic functions.  Little
is known about the index of zeros \emph{on} the critical set, where
the Jacobian vanishes; such zeros are called \emph{singular zeros}.
Our main result is a characterization of the index of singular zeros,
which enables one to determine the index directly from the power series of
$h$.

\end{abstract}

\opt{tandf}{
    \begin{keywords}
}
\opt{preprint}{
\paragraph*{Keywords}
}
Harmonic mappings; Poincar\'e index; singular zero; multiplicity; critical set
\opt{tandf}{
    \end{keywords}
}

\opt{tandf}{
    \begin{classcode}
}
\opt{preprint}{
\paragraph*{AMS Subject Classification (2010)}
}
31A05, 30C55
\opt{tandf}{
    \end{classcode}
}

\opt{mydraft}{
\listoftodos
}

\section{Introduction}

Let $f$ be a harmonic mapping of the complex plane, i.e., $f : D
\subset \C \rightarrow \C$ with $\Delta f = 0$.  Such functions have a
(local) representation $f = h + \conj{g}$, where $h$ and $g$ are
analytic functions.   The functions $f$ and $\conj{g}$ are called
the analytic and anti-analytic parts of $f$, respectively;
see~\cite{Duren2004} for a general introduction.  In this work we
study functions of the type
\begin{equation}
    \label{eq:harm_z}
    f(z) = h(z) - \conj{z},
\end{equation}
that is, where the anti-analytic part simply is $-\conj{z}$.
Functions of this type have been of interest in gravitational
lensing~\cite{MPW1999,KhavinsonNeumann:2006,KhavinsonNeumann:2008,LSL_grg1,SLL_cmft1,LSL_cmft2},
and they also have been studied in the context of Wilmshurst's
conjecture~\cite{Wilmshurst:1998, Geyer:2008, KhavinsonSwiatek2003}.
In all these works the zeros and their indices have a pronounced role.

Zeros of harmonic functions like in~\eqref{eq:harm_z} do not have a
\emph{multiplicity} in the classical sense of polynomials or analytic
functions, but the notion of multiplicity can be generalized to the
change of argument around a zero (or the ``winding'' around it);
see~\cite{Balk1991, Sheil-Small2002}.  We call this ``generalized
multiplicity'' the \emph{index} of the zero.

The \emph{critical set} of~\eqref{eq:harm_z}, i.e., the set where the Jacobian
of $f$ vanishes, divides the complex plane into regions where $f$ is
either sense-preserving, or sense-reversing, depending on the sign of
the Jacobian.  Within these regions the harmonic mapping $f$ is
locally one-to-one, and shares many properties with analytic (or
anti-analytic) functions.  In particular, within these regions, we have
an argument
principle~\cite{DurenHengartnerLaugesen1996,SuffridgeThompson2000}
that allows to count the number of zeros encircled by a curve: indices
of sense-preserving zeros are $+1$, and indices of
sense-reversing zeros are $-1$.  Moreover the index of a zero
$z_0$ can be determined directly from the power series of $h$ at
$z_0$.

In this work we study the index of isolated zeros $z_0$ \emph{on} the
critical set, called \emph{singular zeros}.  Although the index is
defined for these zeros as well (see~\cite{Balk1991,
Sheil-Small2002}), very little is known about it in the existing
literature.  In a first result we show
that the index of $f$ at such a zero can only take values in
$\{-1,0,1\}$, and that every value is attainable, for which we give
examples.
Our main contribution in this work, however, is a
\emph{characterization} of the index
in terms of the power series of $h$ at $z_0$.  The characterization is
almost complete --- except for one curious configuration of the
coefficients of the power series,  which we discuss with great detail
later on.

The organization is as follows.  Section~\ref{sect:background}
contains some background material.  In Section~\ref{sect:bounds} we
derive a bound on the index of singular zeros and present some
examples.  The main Section~\ref{sect:index} contains our
characterization of the index of singular zeros of $f(z) = h(z) - \conj{z}$.
We discuss possible future work in Section~\ref{sect:conclusion}.

\section{Background}
\label{sect:background}

Whether a harmonic function $f = h + \conj{g}$ is sense-preserving or
sense-reversing is determined by the sign of the Jacobian of $f$;
see~\cite{Duren2004}. In our case of interest, where $f(z) = h(z) -
\conj{z}$, a classification can be cast as follows.

\begin{defn} \label{defn:sense-pres-rev}
Let $f(z) = h(z) - \bar{z}$, with an analytic function $h$, and let $z_0 \in 
\C$.  Then
\begin{compactenum}
\item $f$ is called \emph{sense-preserving} at $z_0$ if $\abs{h'(z_0)} > 1$,
\item $f$ is called \emph{sense-reversing} at $z_0$ if $\abs{h'(z_0)} < 1$,
\item $z_0$ is called a \emph{singular point of $f$} if $\abs{h'(z_0)} = 1$.
\end{compactenum}
If additionally $f(z_0) = 0$, the point $z_0$ will be called a
\emph{sense-preserving}, \emph{sense-reversing} or \emph{singular
    zero}, respectively.  If the zero $z_0$ is not singular, we will
    say that $z_0$ is a \emph{regular} zero.
\end{defn}

\subsection{The winding of a function along a curve}

We recall the definition of the winding of a continuous function along
a curve; see~\cite{Balk1991}, \cite[p.~101]{Wegert2012}, or
\cite[p.~29]{Sheil-Small2002}, where the winding is called ``degree''.
Let $\Gamma$ be a curve in the complex plane parametrized by $\gamma :
[a, b] \to \C$, i.e., $\gamma$ is a continuous function.  Throughout
this article we assume that $\Gamma$ is rectifiable.  Let $f : \Gamma
\to \C$ be a continuous function that has no zeros on $\Gamma$, and
denote by $\arg(f \circ \gamma)$ a continuous branch of the argument
of $f \circ \gamma$.  Then the \emph{winding of $f$ on $\Gamma$} is
defined as the change of argument of $f$ along the curve,
\begin{equation*}
V(f; \Gamma) = \frac{1}{2 \pi} [ \arg(f(\gamma(b))) -
    \arg(f(\gamma(a))) ].
\end{equation*}
The winding is independent of the choice of the branch of the
argument, and of the parametrization.  We summarize a few useful
properties of the winding.

\begin{prop}[{see \cite[p.~37]{Balk1991} or~\cite[p.~29]{Sheil-Small2002}}]
\label{prop:winding}
Let $\Gamma$ be a curve, and let $f$ and $g$ be continuous and nonzero 
functions on $\Gamma$.
\begin{enumerate}
\item If $\Gamma$ is a closed curve, then $V(f; \Gamma)$ is an integer.

\item If $\Gamma$ is a closed curve and if there exists a continuous and
single-valued branch of the argument on $f(\Gamma)$, then $V(f; \Gamma) = 0$.


\item We have $V(f g; \Gamma) = V(f; \Gamma) + V(g; \Gamma)$.

\item If $f(z) = c \neq 0$ is constant on $\Gamma$, then $V(f; \Gamma) = 0$.

\end{enumerate}
\end{prop}

\begin{example} \label{ex:multiplicity}
Let $f(z) = (z-z_0)^n$ with $n \in \Z$ and consider the circle $\Gamma
    = \{z \in \C : \abs{z-z_0} = r > 0\}$
parametrized by $\gamma(t) = z_0 + r e^{it}$, $0 \leq t \leq 2 \pi$.
Then $\arg(f(\gamma(t))) = n t$ is a continuous branch of the argument of 
$f \circ \gamma$, which shows that $V(f; \Gamma) = \frac{1}{2 \pi} (2 \pi 
n - 0) = n$.

Now consider $f(z) = (z-z_0)^n g(z)$ where $n \in \Z$ and where $g$ is
analytic and nonzero in a 
disk $D = \{ z : \abs{z-z_0} < R \}$. For $0 < r < R$, 
the closed curve $g \circ \gamma$ does not contain the origin in its interior, 
so 
that $V(g; \Gamma) = 0$.  With Proposition~\ref{prop:winding} we find $V(f; 
\Gamma) = V((z-z_0)^n; \Gamma) + V(g; \Gamma) = n$.  For a zero of $f$ ($n > 
0$) the winding is the multiplicity of the zero.  For a pole of $f$ ($n < 0$), 
the winding is minus the order of the pole.
\end{example}

\begin{example} \label{ex:winding_conjz}
Let $f(z) = \conj{z}$ and $\gamma(t) = r e^{it}$, $0 \leq t \leq 2 \pi$, with 
$r > 0$.  Then $-t$ is a continuous branch of the argument of $f \circ \gamma$, 
showing $V(\conj{z}; \Gamma) = -1$.
\end{example}

We will often show that two functions $f$ and $g$ have the same
winding along a closed curve, and our two main tools for this are
homotopy and Rouch\'e's theorem.  Let $\Gamma$ be a closed curve with
parametrization $\gamma$, then $V(f; \Gamma)$ is the winding number of
the closed curve $f \circ \gamma$. If $f \circ \gamma$ and $g \circ
\gamma$ are homotopic in $\C \setminus \{0\}$, then $V(f; \Gamma)$ =
$V(f; \Gamma)$; see~\cite[p.~88]{Conway1978} 
or~\cite[Lemma~2.7.22]{Wegert2012}.  The symmetric
formulation of Rouch\'es theorem we use is as follows;
see~\cite[Theorem~2.3]{SLL_cmft1}.

\begin{thm}[Rouch\'e's theorem] \label{thm:Rouche}
Let $\Gamma$ be a closed curve, and let $f$ and $g$ be two continuous functions 
on $\Gamma$.  If
\begin{equation*}
\abs{ f(z) + g(z) } < \abs{ f(z) } + \abs{ g(z) }, \quad z \in \Gamma,
\end{equation*}
then $f$ and $g$ have the same winding on $\Gamma$, i.e., $V(f; \Gamma) = V(g; 
\Gamma)$.
\end{thm}

\subsection{The index of a function at a point}

The argument principle connects the global change of argument along a
curve to the local change of argument around a single point.  The
latter is called the Poincar\'e index, or multiplicity of $f$, or
simply ``index'' at the point.

\begin{defn} \label{defn:Poincare_index}
Let $f$ be continuous and nonzero in the punctured disk $\{ z \in \C : 0 < 
\abs{z-z_0} < R \}$.  Let $0 < r < R$ and let $\Gamma_r$ be the positively 
oriented circle with center $z_0$ and radius $r$.
Then the \emph{Poincar\'e index} of $f$ at $z_0$ 
is defined as
\begin{equation*}
\ind(f; z_0) \coloneq V(f; \Gamma_r) \in \Z.
\end{equation*}
The point $z_0$ is called an isolated \emph{exceptional point} of $f$
if it is a zero of $f$, or if $f$ is not continuous at $z_0$, or
if $f$ is not defined at $z_0$.
\end{defn}

The Poincar\'e index is independent of the choice of $r$, and the
circle can even be replaced by an arbitrary positively oriented Jordan
curve that winds around $z_0$; see~\cite[p.~39]{Balk1991}
or~\cite[Section~2.5.1]{Sheil-Small2002}.  The Poincar\'e index
is a generalization of the multiplicity of a zero and order of a
pole of an analytic function; see Example~\ref{ex:multiplicity}.

The only isolated exceptional points of the function $f(z) = h(z) - \conj{z}$, 
where $h$ is analytic, are the zeros of $f$ and the isolated singularities 
of $h$ (poles, removable singularities and essential singularities).

We briefly discuss the connection of the Poincar\'{e} index with phase 
portraits, which are a convenient way to visualize complex 
functions~\cite{Wegert2012,WegertSemmler2011}.  Roughly speaking, each point on 
the unit circle is associated with a color, and the domain of $f$ is colored 
according to the value its phase $f(z)/\abs{f(z)} = \exp(i \arg(f(z)))$ 
takes on the unit circle.
Let $f$ be a continuous complex function.
The Poincar\'e index of an isolated exceptional point $z_0$ of $f$ is the 
change of argument of $f(z)$ while $z$ travels once around $z_0$ on a small 
circle in the positive sense.
This corresponds exactly to the \emph{chromatic number} of $\gamma$, as
discussed in~\cite[p.~772]{WegertSemmler2011}.  Thus, less formally,
the Poincar\'{e} index corresponds to the number of times we run through the 
color wheel while travelling once around $z_0$ in the positive direction, and 
the sign of the Poincar\'{e} index is revealed by the ordering in which the
colors appear.  This observation allows to determine the Poincar\'{e}
index of $f$ at an isolated exceptional point from a phase portrait.

We use the same color scheme for the phase plots as
in~\cite{Wegert2012}.  The color ordering while travelling around some
point $z_0$ is exemplified for the indices $+1$, $+2$, $-1$ and $0$ as
follows (left to right, $z_0$ is indicated by the black dot):
\begin{center}
\hspace*{25pt}
\includegraphics[width=55pt]{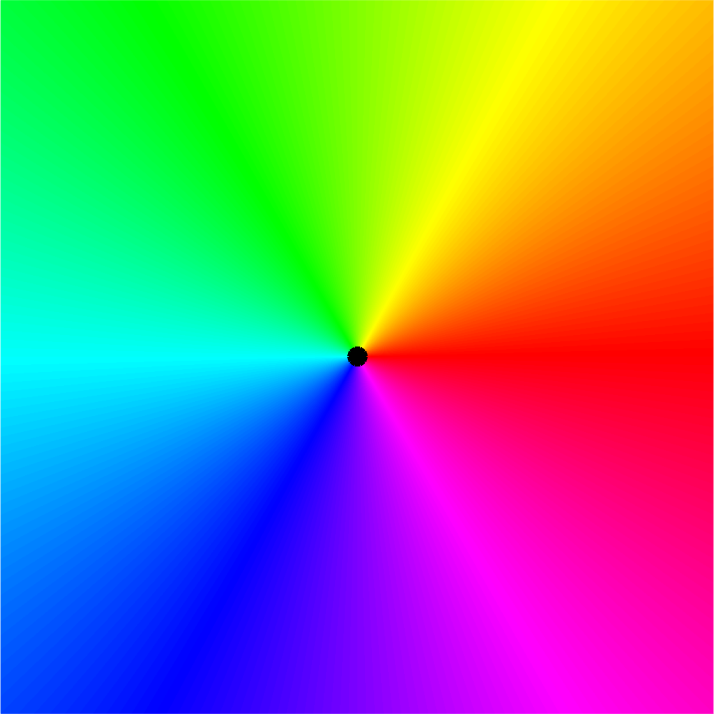}
\hfill
\includegraphics[width=55pt]{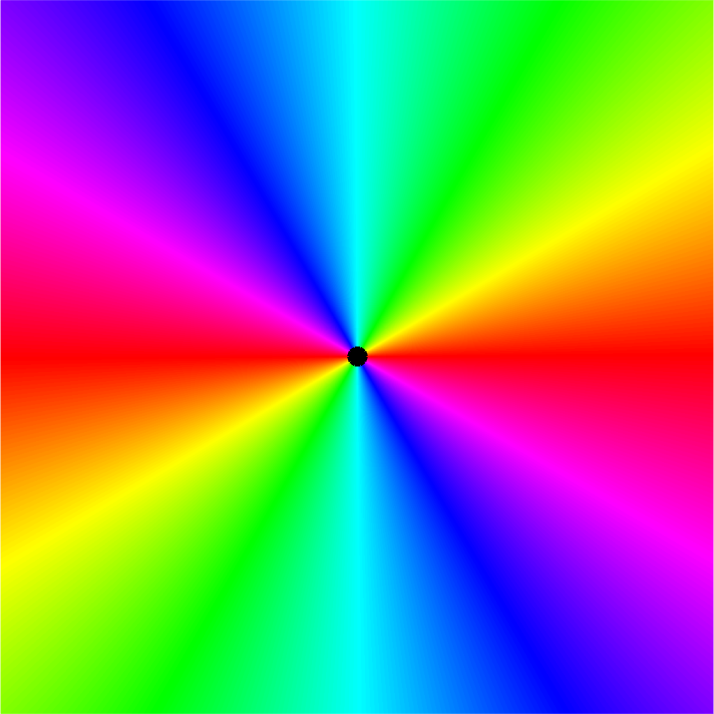}
\hfill
\includegraphics[width=55pt]{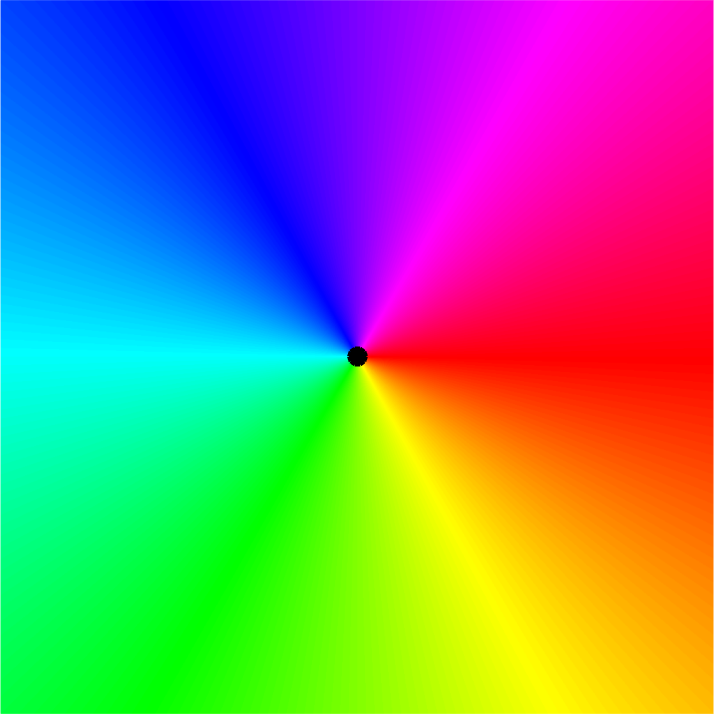}
\hfill
\includegraphics[width=55pt]{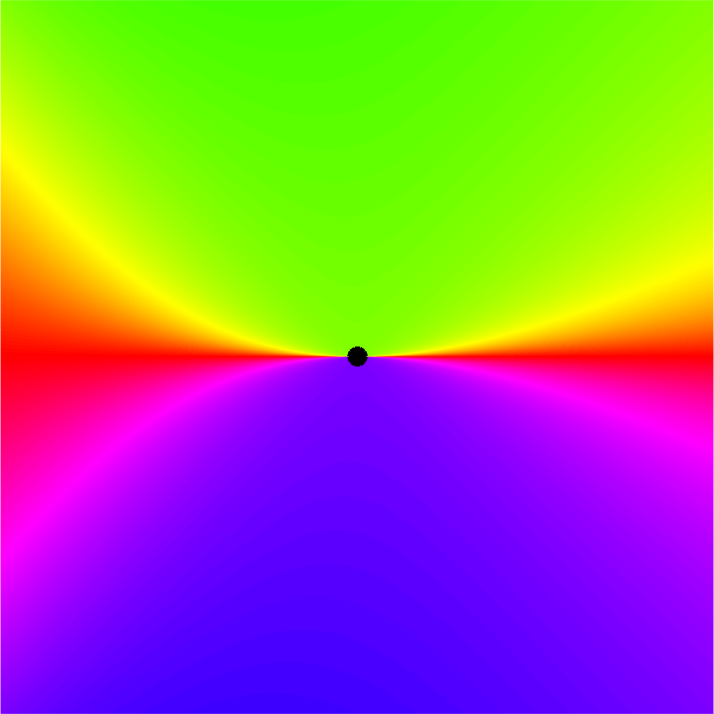}
\hspace*{25pt}
\end{center}
The phase plots in this paper have been generated with a
\textsc{Matlab}\textsuperscript{\textregistered} implementation 
close to~\cite[p.~345]{Wegert2012}.

When $h$ is rational, the index of $f(z) = h(z) - \conj{z}$ at sense-preserving 
and sense-reversing exceptional points is 
known (\cite[Poprosition~2.7]{SLL_cmft1}; see 
Proposition~\ref{prop:regular} below for the generalization to analytic $h$).
There are no such previous results for singular zeros.

\begin{prop} \label{prop:index_rational}
Let $f(z) = h(z) - \conj{z}$ where $h$ is a rational function of degree at 
least $2$.
\begin{enumerate}
\item If $z_0$ is a pole of $h$ of order $m$, then $\ind(f; z_0) = -m$.
\item If $z_0$ is a sense-preserving zero of $f$, then $\ind(f; z_0) = +1$.
\item If $z_0$ is a sense-reversing zero of $f$, then $\ind(f; z_0) = -1$.
\end{enumerate}
\end{prop}

The argument principle connects the global change of argument along a curve 
to the local change of argument around exceptional points.
Here we state a version for merely continuous complex functions.

\begin{thm}[{\cite[p.~39]{Balk1991}, \cite[p.~44]{Sheil-Small2002}}] 
\label{thm:argument_principle}
Let the function $f$ be continuous in the closed region $\overline{D}$
limited by the closed Jordan curve $\Gamma$ and suppose that $f$ has only a 
finite number of exceptional points $z_1, z_2, \ldots, z_n$ in $\overline{D}$ 
neither of which is on $\Gamma$.  Then
\begin{equation*}
V(f; \Gamma)
= \ind(f; z_1) + \ind(f; z_2) + \ldots + \ind(f; z_n).
\end{equation*}
\end{thm}

When $f$ is analytic, the argument principle allows us to count the number of 
zeros of $f$ interior to $\Gamma$.  For harmonic functions the winding along 
the boundary also is the sum of the indices.  The difference is, however, 
that the index of $f$ at a zero may be positive, negative or even zero.
See also the discussion in~\cite[p.~46]{Sheil-Small2002}.

We now collect a few useful facts which we will use in conjunction with the
argument principle and Proposition~\ref{prop:index_rational}
to compute the index of a singular zero.  For a
rational function $r = p/q$ we will say that it
is of the type $(\deg(p), \deg(q))$.

\begin{prop} \label{prop:global_winding}
Let $f(z) = h(z) - \conj{z}$, where $h$ is a rational function of
degree at least two.  Then there exists a $R > 0$ such that all
zeros of $f$ and all poles of $h$ are in the interior of the
circle $\Gamma = \{z \in \C : \abs{z} = R\}$.
\begin{enumerate}
\item If $h$ is of type $(j, n)$ with $j \leq n$, we have $V(f; \Gamma) = -1$.
\item If $h$ is of type $(k+n, k)$ with $n \geq 2$, we have $V(f; \Gamma) = 
n$.
\item If $h$ is a polynomial of degree $n \geq 2$, then 
$V(f; \Gamma) = n$.
\end{enumerate}
\end{prop}

\begin{proof}
The function $f(z) = h(z) - \conj{z}$ with rational $h$ of degree $n \geq 2$ 
has at most $5(n-1)$ zeros~\cite{KhavinsonNeumann:2006}, so that $R$ is well 
defined.  Since there are no exceptional points outside $R$, we can enlarge $R$ 
as needed without changing the winding.

To prove the first assertion, consider
\begin{equation*}
f(z) = h(z) - \conj{z} = -\conj{z} ( 1 - h(z) / \conj{z}).
\end{equation*}
Since $\abs{h(z) / \conj{z}} \to 0$ as $\abs{z} \to \infty$, the function $1 - 
h(z)/\conj{z}$ takes on values in a disk around $1$ that does not contain the 
origin, provided $R$ is sufficiently large.  We then find
\begin{equation*}
V(f; \Gamma) = V(\conj{z}; \Gamma) + V(1 - h(z)/\conj{z}; \Gamma) = 
V(\conj{z}; \Gamma) = -1,
\end{equation*}
where we used Proposition~\ref{prop:winding} and 
Example~\ref{ex:winding_conjz}.
Part two is proved similarly by factoring out $z^n$, which is the largest term 
as $\abs{z} \to \infty$.  Part three is a special case of part two.
\end{proof}

Finally, recall Landau's $\bigO$-notation.  We write $g(z) \in \bigO(z^n)$ when 
$g(z) / z^n$ is bounded for $z \to 0$.  In this article the function $g$ will 
always be analytic in a neighbourhood of the origin.
We write $f(z) + \bigO(z^n)$ for an expression $f(z) + g(z)$ with $g(z) \in 
\bigO(z^n)$.

\section{Index bounds}
\label{sect:bounds}

We derive a bound for the index of an isolated \emph{singular} zero.
A similar bound has been obtained for harmonic polynomials
in~\cite[p.~66]{Sheil-Small2002}, and a lower bound for polyanalytic
functions is given in~\cite[Corollary~2.9]{Balk1991}.  For
completeness, we also give the corresponding result for regular zeros.

\begin{thm} \label{thm:index_bound}
Let $z_0$ be an isolated zero of $f(z) = h(z) - \conj{z}$, where $h$
    is an analytic function.
    \begin{enumerate}
        \item If $z_0$ is sense-reversing, then $\ind(f; z_0) = -1$.
        \item If $z_0$ is sense-preserving, then $\ind(f; z_0) = +1$.
        \item If $z_0$ is singular, then $\ind(f; z_0) \in \{-1,0,+1\}$.
    \end{enumerate}
\end{thm}
\begin{proof}
Throughout we assume $z_0 = 0$ (otherwise we substitute $w = z - z_0$).

If $z_0$ is a regular zero, i.e., either sense-preserving or sense-reversing,
then the Jacobian $J_f(z_0) = \abs{h'(z_0)}^2 - 1$ of $f$ at $z_0$ is nonzero, 
so that $f$ is locally one-to-one (injective).
Thus $f$ maps a sufficiently small circle around $z_0$ to a Jordan curve containing the 
origin in its interior, from which it follows that the winding of $f$ along the 
circle is $+1$ (if $f$ sense-preserving at $z_0$) or $-1$ (if $f$ sense-reversing at 
$z_0$).

If $z_0 = 0$ is a singular zero of $f$, we have $h(0) = 0$ and $\abs{h'(0)} = 1$, so that the analytic
function $h$ has a
simple zero at $0$.  Let $R > 0$ such that $h$ is analytic in
    $D \defby \{z \in \C : \abs{z} < R\}$ and such that $z_0 = 0$ is the only zero of $f$
inside the circle of radius $R$ around the origin.
The function $z h(z)$ is analytic in $D$, and has a double
zero at $0$.  By~\cite[Theorem~3.4.11]{Wegert2012} (``$zh(z)$ is
locally bi-valent''), there exist $\varepsilon, \delta > 0$ such
that any $w$ satisfying $\abs{w} < \delta$ has two preimages in the disk $\{z \in \C
    : \abs{z}
< \varepsilon$\}.  Let $0 < r < \min\{R, \varepsilon,
\sqrt{\delta}\}$
and write
\begin{equation*}
h(z) - \conj{z} = \frac{1}{z} ( z h(z) - \abs{z}^2 ),
\end{equation*}
which does not vanish on $\Gamma = \{z \in \C : \abs{z} = r\}$,
since by construction $0$ is the only zero of $f(z) = h(z) - \conj{z}$
in $D$.  This allows us to compute
\begin{equation*}
    \ind(f;0) = V(f; \Gamma) = V(z^{-1}; \Gamma) + V(z h(z) - r^2; \Gamma) = -1 + V(zh(z)-r^2; 
\Gamma)
\end{equation*}
and it remains to show $0 \leq V(z h(z) - r^2; \Gamma) \leq 2$.
Since $z h(z) - r^2$ is analytic in $D$, $V(zh(z)-r^2; 
\Gamma) \geq 0$ is the number of zeros of $z h(z) - r^2$ interior to $\abs{z} = 
r$ by the argument principle for analytic functions.  Since $r^2 <
    \delta$, there are exactly two zeros in $\{z \in \C : \abs{z} <
    \varepsilon\}$, and hence at most two such zeros in the smaller
    disk of radius $r$. It follows that  $V(z h(z) - r^2; \Gamma) \le 2$.
\end{proof}

\subsection{Examples}
\label{sect:examples}

In Theorem~\ref{thm:index_bound} we have seen that functions of the
type $h(z) - \conj{z}$ have index $\{-1, 0, 1\}$ at a singular zero.
All three cases occur, and we give an explicit example for each.  The
four examples we consider are illustrated in
Figure~\ref{fig:sing_examples}.  While the index is easily spotted
from these phase portrait, computing the index is much more involved,
as we will see.

\begin{figure}[t!]
\begin{center}
\includegraphics[width=.49\textwidth]{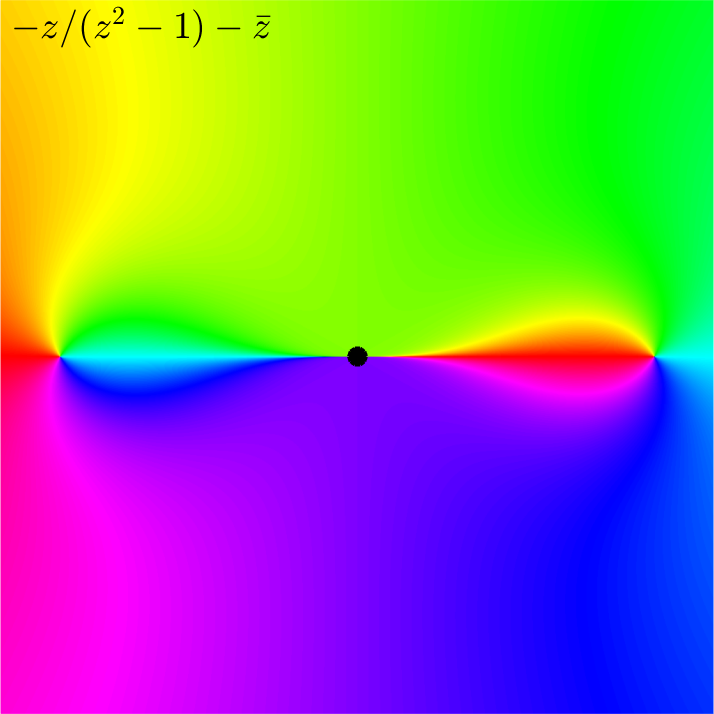} \hfill
\includegraphics[width=.49\textwidth]{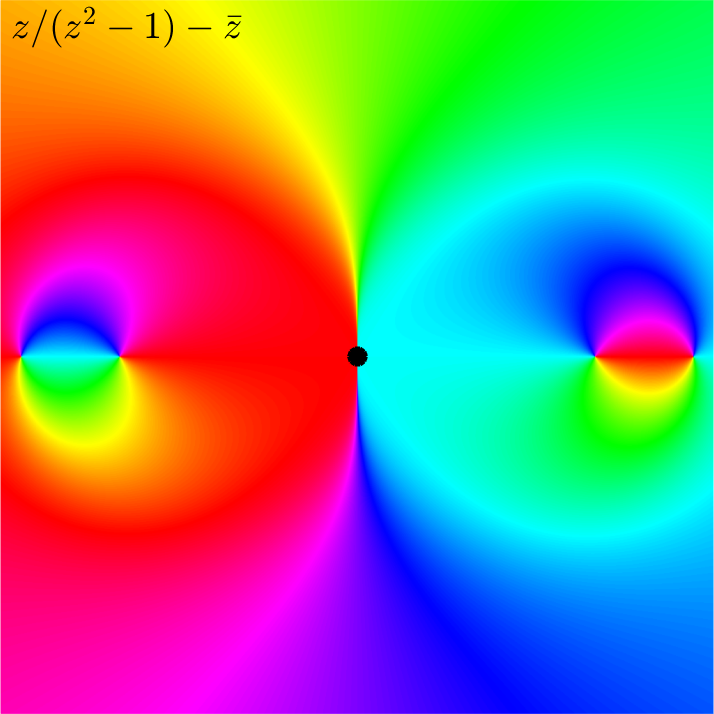} \\[5pt]
\includegraphics[width=.49\textwidth]{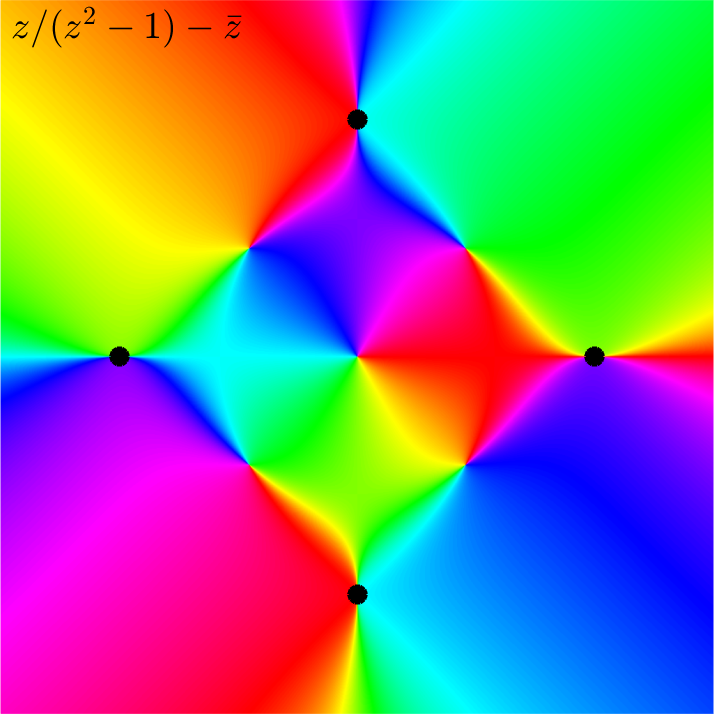} \hfill
\includegraphics[width=.49\textwidth]{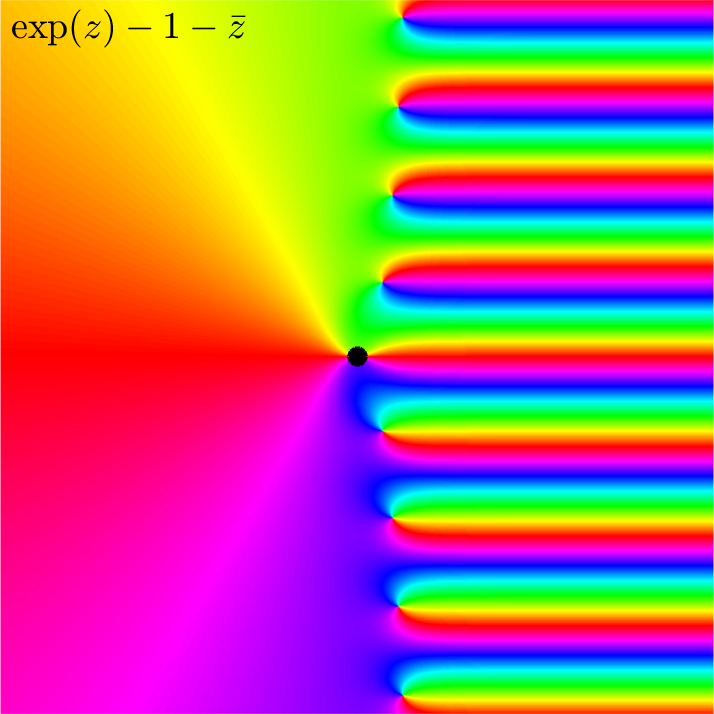}
\end{center}
    \caption{Phase portraits for the function in
    Examples~\ref{ex:index+1}--\ref{ex:exp}.  The singular zeros
    discussed in these examples are marked with black
    dots.  The indices of the marked zeros are $+1$ (top left),
    $-1$ (top right) and $0$ (bottom row). \label{fig:sing_examples}}
\end{figure}

\begin{example} \label{ex:index+1}
We show that $f(z) = h(z) - \conj{z}$ with $h(z) = -z/(z^2 - 1)$ 
has a singular zero with Poincar\'{e} index $+1$.

We compute the zeros of $f$.  We have that $f(z) = 0$ is equivalent to
\begin{equation}
\frac{z^2}{1 - z^2} = \abs{z^2}, \label{eqn:zeros_of_f}
\end{equation}
implying $z^2 = \tfrac{ \abs{z}^2 }{ 1 + \abs{z}^2 } \geq 0$. 
Then~\eqref{eqn:zeros_of_f} is equivalent to $z^2 = z^2 (1-z^2)$, showing
that $z_0 = 0$ is the only zero of $f$.  Since
\begin{equation}
\abs{h'(z)} = \abslr{\frac{z^2 + 1}{(z^2-1)^2}}, \label{eqn:h1_prime}
\end{equation}
we have $\abs{h'(0)} = 1$, so that $z_0$ is a singular zero of $f$.

Further, $f$ has the two simple poles $\pm 1$ with $\ind(f; \pm 1) = -1$;
see Proposition~\ref{prop:index_rational}.  Combining 
Proposition~\ref{prop:global_winding} and Theorem~\ref{thm:argument_principle},
we find on a sufficiently large circle $\Gamma$
\begin{equation*}
-1 = V(f; \Gamma) = \ind(f; -1) + \ind(f; 0) + \ind(f; 1)
= -2 + \ind(f; 0),
\end{equation*}
hence $\ind(f; 0) = +1$.  Thus $0$ is a singular zero with Poincar\'e index
$+1$.
\end{example}

\begin{example}
\label{ex:index-1}
We show that $f(z) = h(z) - \bar{z}$ with $h(z) = {z}/(z^2 - 1)$ 
has a singular zero with Poincar\'e index $-1$.

Let us compute the zeros of $f$.  Clearly, $z_0 = 0$ is a zero of
$f$.  For $z \neq 0$, we have that $f(z) = 0$ is equivalent to
$(\abs{z}^2 - 1) z^2 = \abs{z}^2$.  Writing $z = \rho e^{i \varphi}$ with
$\rho > 0$ and $\varphi \in \R$, this is equivalent to
$(\rho^2 - 1) e^{i 2 \varphi} = 1$.
In particular, $\rho \neq 1$ and $e^{i 2 \varphi} \in \R$, so that either
$e^{i 2 \varphi} = -1$ or $e^{i 2 \varphi} = +1$.  In the first case $e^{i 2
\varphi} = -1$ implies the contradiction $\rho = 0$.  In the second case 
$e^{i 2 \varphi} = 1$ implies $\varphi = 0, \pi$ and we find $\rho = \sqrt{2}$. 
Hence $f$ has the three zeros $0, \pm \sqrt{2}$.
Since $\abs{h'(z)}$ is given by~\eqref{eqn:h1_prime} we have
\begin{equation*}
\abs{h'(0)} = 1 \quad \text{and} \quad
\abs{h'(\pm \sqrt{2})} = 3 > 1,
\end{equation*}
so that $0$ is a singular zero of $f$, and $\pm \sqrt{2}$ are
sense-preserving zeros of $f$ and thus have Poincar\'{e} index $+1$ by
Proposition~\ref{prop:index_rational}.

Further $f$ has the two simple poles $\pm 1$ with $\ind(f; \pm 1) = -1$;
see Proposition~\ref{prop:index_rational}.  Combining
Proposition~\ref{prop:global_winding} and the argument principle as in
Example~\ref{ex:index+1}, we find that $\ind(f; 0) = -1$. 
Thus, $z_0 = 0$ is a singular zero of $f$ with Poincar\'e index $-1$.
\end{example}

\begin{example}
\label{ex:index0}
We show that $f(z) = h(z) - \bar{z}$ with $h(z) = 2 z^3 + \tfrac{1}{8 z}$ 
has singular zeros with Poincar\'e index $0$.

We compute the zeros of $f$.  Clearly, $z = 0$ is a pole of $h$ and
thus not a zero of $f$.  Then, $f(z) = 0$ is equivalent to
$2 z^4 + \tfrac{1}{8} - \abs{z}^2 = 0$.
Writing $z = \rho e^{i \varphi}$, where $\rho > 0$ and $\varphi \in \R$, we
find
\begin{equation}
2 \rho^4 e^{i 4 \varphi} + \tfrac{1}{8} - \rho^2 = 0, \label{eqn:polar_zeros_f3}
\end{equation}
so that $e^{i 4 \varphi}$ is real.

Consider first the case $e^{i 4 \varphi} = -1$.  We then have
$\varphi = \tfrac{(2k+1) \pi}{4}$ with $k \in \Z$,
and~\eqref{eqn:polar_zeros_f3} becomes $2 \rho^4 + \rho^2 - \tfrac{1}{8} = 0$,
implying $\rho^2 = \tfrac{\sqrt{2}-1}{4}$.  Thus, $f$ has the four zeros
$z_k = (\tfrac{\sqrt{2}-1}{4})^{\frac{1}{2}} e^{i \frac{(2k+1) \pi}{4}}$, $k =
4, 5, 6, 7$.  We determine their type.
We have $h'(z) = 6 z^2 - \tfrac{1}{8 z^2}$.  A short computation yields
$h'(z_k) = ( 2 \sqrt{2} - 1 ) (-1)^k i$.  Thus $\abs{h'(z_k)} > 1$ and the
$z_k$ are sense-preserving zeros of $f$.  By
Proposition~\ref{prop:index_rational} we have $\ind(f; z_k) = +1$.

In the second case, $e^{i 4 \varphi} = +1$, we have $\varphi = k \tfrac{\pi}{2}$
for some $k \in \Z$, and equation~\eqref{eqn:polar_zeros_f3} becomes
$0 = 2 \rho^4 - \rho^2 + \tfrac{1}{8} = 2 (\rho^2 - \tfrac{1}{4})^2$, which
yields $\rho = \tfrac{1}{2}$.  Thus, $f$ has another four zeros $z_k =
\tfrac{1}{2} e^{i k \frac{\pi}{2}}$, $k = 0, 1, 2, 3$.  A short computation
shows $h'(z_k) = (-1)^k$, so that these $z_k$ are singular zeros of $f$.
We show that $z_0, z_1, z_2$, and $z_3$ have the same Poincar\'e index. 
Note that $f(z) = e^{i \frac{\pi}{2}} f( e^{i \frac{\pi}{2}} z)$ holds for
all $z$.  Denote by $\Gamma_0$ a small circle centered at $z_0$ suitable for the
computation of $\ind(f; z_0)$, recall Definition~\ref{defn:Poincare_index}.
Fix $k \in \{ 1, 2, 3 \}$.  Set $\Gamma_k \coloneq e^{i k \tfrac{\pi}{2}}
\Gamma_0$, which then is a circle centered at $z_k$ with $\ind(f; z_k) =
V(f; \Gamma_k)$.  We find
\begin{equation*}
\begin{split}
\ind(f; z_0) &= V(f; \Gamma_0)
= V( e^{i k \frac{\pi}{2}} f( e^{i k \frac{\pi}{2}} z); \Gamma_0)
= V( f( e^{i k \frac{\pi}{2}} z); \Gamma_0) \\
&= V( f(z); \Gamma_k) = \ind(f; z_k).
\end{split}
\end{equation*}
Thus the singular zeros $z_0, z_1, z_2$, and $z_3$ of $f$ all have the same
Poincar\'e index.

Clearly, $0$ is the only pole of $f$ and has $\ind(f; 0) = -1$.  Now,
since $h(z) = (16 z^4 + 1)/(8z)$ is of type $(4, 1)$, we find by
Proposition~\ref{prop:global_winding} and the argument principle
\begin{equation*}
3 = \ind(f; 0) + \sum_{k=0}^7 \ind(f; z_k) = 3 + 4 \ind(f; z_0).
\end{equation*}
Hence $\ind(f; z_k) = 0$ for $k = 0, 1, 2, 3$, and we have shown that $f$
has singular zeros with Poincar\'e index $0$.
\end{example}

\begin{example}
\label{ex:exp}
Let $f(z) = h(z) - \conj{z} = e^z - 1 - \conj{z}$, which has an isolated zero 
at the origin.  Since $\abs{h'(0)} = e^0 = 1$, it is a singular zero.
The phase portrait (Figure~\ref{fig:sing_examples}) indicates that $f$ has 
index $0$, but determining the index as in the previous examples is not 
possible.
\end{example}

All these computations were quite laborious.  With our main results
Theorems~\ref{thm:index_sing_zero}
and~\ref{thm:index_sing_zero_general} below, we will be able to compute
these indices very easily by looking at the power series of the
analytic part; see Section~\ref{sect:examples2}.

\section{Determining the index from the power series}
\label{sect:index}

Let $f(z) = h(z) - \conj{z}$ be a complex-valued harmonic function,
where $h$ is an analytic function.  We aim to
characterize the index of $f$ at an isolated zero $z_0$ by the
coefficients of the Taylor series of $h$ at $z_0$.
For regular zeros $z_0$, the index can be easily inferred from the
series, which is a direct consequence of
Definition~\ref{defn:sense-pres-rev}, and
Theorem~\ref{thm:index_bound}.

\begin{prop} \label{prop:regular}
Let $f(z) = h(z) - \conj{z}$, with $h$ analytic, have a zero at $z_0 \in \C$,
so that
\begin{equation*}
f(z) =  h(z) - \conj{z} = \sum_{k=1}^\infty a_k (z-z_0)^k - \conj{z-z_0}
\end{equation*}
near $z_0$.
\begin{enumerate}
\item If $\abs{a_1} > 1$ the function $f$ is sense-preserving at $z_0$ and $\ind(f; z_0) = +1$.
\item If $\abs{a_1} < 1$ the function $f$ is sense-reversing at $z_0$ 
and $\ind(f; z_0) = -1$.
\end{enumerate}
\end{prop}

\subsection{Determining the index of singular zeros}

The preceding proposition shows that the index of a regular zero is
determined by the first term of the Taylor series of $h$.  The case of
a singular zero is more subtle than that of a regular zero, and the
occurrence of a singular zero is typically excluded in the published
literature on harmonic mappings.  However, as we will see in the
following theorem, the index of a singular zero is determined from the
leading terms in the Taylor series of $h$ as well.

\begin{thm}
\label{thm:index_sing_zero}
Let the complex-valued harmonic function
\begin{equation*}
f(z) = h(z) - \conj{z} = z + \sum_{k=2}^\infty a_k z^k - \conj{z}
\end{equation*}
have an isolated singular zero at the origin.  Let $n \geq 2$ be the smallest 
index with $a_n \neq 0$, then
\begin{equation*}
\ind(f; 0) =
\begin{cases}
    \phantom{+} 0 & \re(a_n) \neq 0, \; n \, \text{even} \\
    +1 & \re(a_n)  > 0, \; n \,\text{odd} \\
    -1 & \re(a_n)  < 0, \; n \, \text{odd}
\end{cases}.
\end{equation*}
\end{thm}

\begin{rem}
\label{rem:comment1}
Recall from Proposition~\ref{prop:regular} that the index of a regular zero
$z_0$ of $h(z) - \conj{z}$ is completely determined by the \emph{first}
derivative of $h$ at $z_0$.  If $z_0$ is singular, the index is, except for
the case $\re(a_n) = 0$, completely determined by the \emph{first two}
non-vanishing derivatives of $h$ at $z_0$.
\end{rem}

\begin{rem}
The only case not covered by Theorem~\ref{thm:index_sing_zero} is that of
a purely imaginary coefficient $a_n$.  No characterization of the index is 
given in this case (of course it still is in $\{0, \pm 1\}$).  We have 
constructed examples showing that in this
particular case the first two non-vanishing derivatives are
\emph{not} sufficient to characterize the index at $0$.
This behaviour is illustrated in Section~\ref{sect:imag_diff}.
Hence a complete characterization requires local or global information about 
$h$ beyond of what is used in Theorem~\ref{thm:index_sing_zero}.
\end{rem}

\begin{rem}
\label{rem:comment3}
For harmonic polynomials of the special form
\begin{equation*}
f(z) = z + a_n z^n - \conj{z}, \quad a_n \neq 0, 
\end{equation*}
we give a complete characterization of the index of $f$ at $0$ in 
Lemma~\ref{lem:compute_index}, that is, including the aforementioned case 
$\re(a_n) = 0$.  As it turns out, the index in $0$ is always $+1$ for these 
functions if $a_n$ is purely imaginary.
\end{rem}

Theorem~\ref{thm:index_sing_zero} makes two normalization assumptions,
neither of which is restrictive.  The first one asserts that the
singular zero of interest is in the origin.  If $z_0$ is \emph{any} zero of
$f$, then $0 = h(z_0) - \conj{z_0}$, and expanding $h$ in a power
series yields
\begin{equation*}
f(z) = \sum_{k=1}^\infty a_k (z-z_0)^k - \conj{z-z_0}.
\end{equation*}
After the change of variables $w = z - z_0$, the zero is 
at the origin, and the index does not change.

The second normalization is $a_1 = 1$.  To see that this is not
a restriction, consider for an arbitrary $a_1 \neq 0$ the function
\begin{equation*}
f(z) = h(z) - \conj{z} = a_1 z + \sum_{k=2}^\infty a_k z^k - \conj{z},
\end{equation*}
and suppose that it has an isolated singular zero at $z_0 = 0$.  This
implies $\abs{a_1} = \abs{h'(0)} = 1$, so that $a_1 = e^{i \theta}$ for some 
$\theta \in \R$.  From
\begin{equation*}
e^{- i \theta/2} f(z) = e^{i \theta/2} z + \sum_{k=2}^\infty a_k e^{- i (k+1) 
\theta/2} (e^{i \theta/2} z)^k - \conj{e^{i \theta/2} z}
\end{equation*}
we find
\begin{equation*}
V(f; \Gamma) = V(e^{-i \theta/2} f(z); \Gamma)
= V(e^{-i \theta/2} f(e^{- i \theta/2} z); \Gamma)
\end{equation*}
for all sufficiently small circles $\Gamma$.  Therefore $f$ and
$z + \sum_{k=2}^\infty a_k e^{- i (k+1) \theta/2} z^k - \conj{z}$
have the same index at the origin, and we can assume $a_1 = 1$.
Substituting back, we can reformulate
Theorem~\ref{thm:index_sing_zero} without the discussed
normalizations.

\begin{thm}
\label{thm:index_sing_zero_general}
Let the complex-valued harmonic function
\begin{equation*}
f(z) = h(z) - \conj{z}
\end{equation*}
with analytic $h$
have an isolated singular zero at $z_0 \in \C$, and $h'(z_0) = e^{i\theta}$.
Let $n \geq 2$ be the smallest index with $h^{(n)}(z_0) = c e^{i\varphi} \neq 
0$, where $c>0$, and set
\begin{equation}
\label{eq:eta}
\eta \coloneq \cos(\varphi - \tfrac{n+1}{2} \theta),
\end{equation}
then
\begin{equation*}
\ind(f; z_0) =
\begin{cases}
\phantom{+} 0 & \eta \neq 0, \; n \text{ even} \\
+1 & \eta > 0, \; n \text{ odd} \\
-1 & \eta < 0, \; n \text{ odd}
\end{cases}.
\end{equation*}
\end{thm}

\subsection{Proof of Theorem~\ref{thm:index_sing_zero}}

Throughout we consider the complex-valued harmonic function
\begin{equation*}
f(z) = h(z) - \conj{z} = z + \sum_{k=2}^\infty a_k z^k - \conj{z},
\end{equation*}
which has a singular zero at the origin, and we assume that the zero
is
isolated.  Then not all coefficients $a_k$ with $k \geq 2$ can
vanish: If $a_k = 0$ for all $k \geq 2$, then $f(z) = z - \conj{z}$,
which vanishes on the whole real line, so that $0$ is not an isolated zero.
Further we denote by $n \geq 2$ the smallest integer satisfying $a_n
\neq 0$.  We thus have
\begin{equation}
    \label{eq:f_series}
f(z) = z + a_n z^n + \sum_{k=n+1}^\infty a_k z^k - \conj{z}, \quad a_n \neq 0.
\end{equation}
The proof of Theorem~\ref{thm:index_sing_zero} is divided in several steps, and 
an outline is as follows.

\begin{enumerate}
\item We show that $f$ and $z + a_n z^n - \conj{z}$ have the same index at $0$ 
(Lemma~\ref{lem:cutoff_reneq0}).
\item We show that $a_n$ can be reduced to one of $\pm 1, \pm i$ 
(Lemmas~\ref{lem:rays} and~\ref{lem:halfplanes}).
\item We show that then $n$ can be replaced by $2$ when $n$ is even and by 
$3$ when $n$ is odd (Lemma~\ref{lem:nreduction}).
\item We explicitly compute the index in each of these cases
(Lemma~\ref{lem:compute_index}).
\end{enumerate}

We begin by showing that the series~\eqref{eq:f_series} can be
truncated after the first non-vanishing coefficient without changing
the index (Lemma~\ref{lem:cutoff_reneq0} below).
For this we need the following technical, preparatory lemma.

\begin{lem} \label{lem:ming_reanneq0}
Let
\begin{equation*}
g(z) = a z^n + z - \conj{z}
\end{equation*}
where $n \geq 2$ and $a \neq 0$, and let $0 < c < 1$.  If $\re(a) \neq 0$, then
\begin{equation}
\abs{g(z)} > c \abs{\re(a)} \rho^n \quad \text{on} \quad \abs{z} = \rho
\label{eqn:ming_reanneq0}
\end{equation}
and if $\re(a) = 0$
\begin{equation}
\abs{g(z)} > c \frac{n}{2} \abs{a}^2 \rho^{2n-1} \quad \text{on} \quad \abs{z} 
= \rho
\label{eqn:ming_rean0}
\end{equation}
for all sufficiently small $\rho$.
\end{lem}

\begin{proof}
Fix $0 < c < 1$,
and write $z = \rho e^{it}$ with $\rho > 0$ and $t \in [0, 2\pi]$.  Then
\begin{equation*}
\abs{g(z)} = \abs{a \rho^n e^{int} + 2 i \rho \sin(t)}
= \rho^n \abs{ a e^{int} + \frac{2i}{\rho^{n-1}} \sin(t) }.
\end{equation*}
The absolute value on the right hand side is the distance between the curve 
$\sigma(t) = \frac{2i}{\rho^{n-1}} \sin(t)$ and the circle $\gamma(t) = - a 
e^{int}$.
We show that these curves have always the required distance.
Define $\delta > 0$ by $\delta = c \abs{\re(a)}$ for $\re(a) \neq 0$ and 
$\delta = c \frac{n}{2} \abs{a}^2 \rho^{n-1}$ if $\re(a) = 0$.  Then $\delta < 
\abs{a}$ for sufficiently small $\rho$.

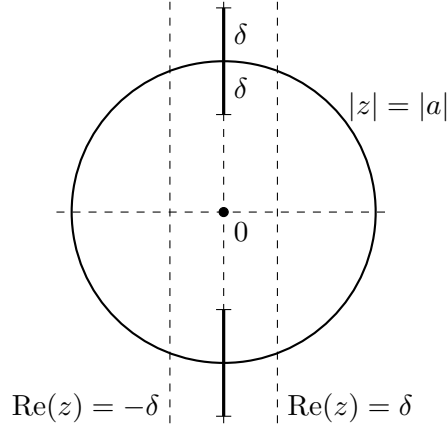
\begin{figure}
\begin{center}
\begin{tikzpicture}[scale = 2]
\draw [dashed] (-1.1, 0) -- (1.1, 0);
\draw [dashed] (0, -1.4) -- (0, 1.4);
\draw [fill=black] (0, 0) circle [radius = 0.03];
\node[below right] at (0, 0) {$0$};
\draw [thick] (0,0) circle [radius = 1];
\node [right] at (0.75, 0.7) {$\abs{z} = \abs{a}$};
\draw [dashed] (-0.3536, -1.4) -- (-0.3536, 1.4);
\draw [dashed] (0.3536, -1.4) -- (0.3536, 1.4);
\node[left] at (-0.3536, -1.3) {$\re(z) = - \delta$};
\node[right] at (0.3536, -1.3) {$\re(z) = \delta$};
\draw[very thick] (0, 1 - 0.3536) -- (0, 1+0.3536);
\draw[very thick] (0, -1 - 0.3536) -- (0, -1+0.3536);
    \draw (-0.05, 1-0.3536) -- (0.05, 1-0.3536);
    \draw (-0.05, 1+0.3536) -- (0.05, 1+0.3536);
    \draw (-0.05, -1-0.3536) -- (0.05, -1-0.3536);
    \draw (-0.05, -1+0.3536) -- (0.05, -1+0.3536);
    \node[right] at (0,1.18) {$\delta$};
    \node[right] at (0,1-0.18) {$\delta$};
\end{tikzpicture}
\end{center}
\caption{Illustration for the proof of Lemma~\ref{lem:ming_reanneq0}.  The thick 
line 
is the location of the sine term for $t$ 
satisfying~\eqref{eqn:tinterval_plus}.}
\label{fig:ming_reanneq0}
\end{figure}

Let first $t \in [0, \pi/2]$.  Then $\sigma(t)$ moves from $0$ to $2 i / 
\rho^{n-1}$ and thus crosses the circle with center $0$ and radius $\abs{a}$;
see Figure~\ref{fig:ming_reanneq0}.
We determine at which time $t \in [0, \pi/2]$ we have
\begin{equation}
\frac{2}{\rho^{n-1}} \sin(t) = \abs{a} + \tau, \quad - \delta \leq \tau \leq 
\delta.
\label{eqn:tinterval_plus}
\end{equation}
For $t \in [0, \pi/2]$ not satisfying~\eqref{eqn:tinterval_plus} we have 
$\abs{\sigma(t) - \gamma(t)} > \delta$, so that~\eqref{eqn:ming_reanneq0} 
or~\eqref{eqn:ming_rean0} is satisfied.
Equation~\eqref{eqn:tinterval_plus} is equivalent to
\begin{equation*}
\sin(t) = \frac{\abs{a} + \tau}{2} \rho^{n-1},
\quad - \delta \leq \tau \leq \delta.
\end{equation*}
For sufficiently small $\rho$ the values are close to zero so that
\begin{equation}
t = \frac{\abs{a} + \tau}{2} \rho^{n-1} + \bigO(\rho^{3n-3}),
\quad - \delta \leq \tau \leq \delta.
\label{eqn:tinterval_plus2}
\end{equation}

Since $\abs{g(z)} \geq \abs{\re(g(z))}$, we bound the real part of $g$:
\begin{equation*}
\re(g(z)) = \re(a \rho^n e^{int}) = \re(a) \rho^n \cos(nt) - \im(a) \rho^n 
\sin(nt).
\end{equation*}
Note that $t$ in~\eqref{eqn:tinterval_plus2} is of order $\bigO(\rho^{n-1})$, 
so 
that
\begin{equation}
\re(g(z)) = \re(a) \rho^n - \im(a) n \rho^n t + \bigO(\rho^{2n}). 
\label{eqn:reg_approx}
\end{equation}
Let $\re(a) \neq 0$.  Inserting $t$ from~\eqref{eqn:tinterval_plus2} 
in~\eqref{eqn:reg_approx} gives
\begin{equation*}
\re(g(z)) = \re(a) \rho^n + \bigO(\rho^{2n-1}),
\end{equation*}
so that
\begin{equation*}
\abs{g(z)} = \abs{\re(a)} \rho^n + \bigO(\rho^{2n-1}) > c \abs{\re(a)} \rho^n
\end{equation*}
for all sufficiently small $\rho$.
Inserting $t$ from~\eqref{eqn:tinterval_plus2} in~\eqref{eqn:reg_approx} when 
$\re(a) = 0$ gives
\begin{equation*}
\re(g(z)) = - \im(a) \frac{n}{2} (\abs{a} + \tau) \rho^{2n-1} + 
\bigO(\rho^{2n}),
\end{equation*}
and
\begin{equation*}
\begin{split}
\abs{g(z)}
&\geq \abs{\im(a)} \frac{n}{2} ( \abs{a} + \tau ) \rho^{2n-1} + \bigO(\rho^{2n}) 
\\
&\geq \abs{\im(a)} \frac{n}{2} (\abs{a} - \delta) \rho^{2n-1} + \bigO(\rho^{2n})
> c \frac{n}{2} \abs{a}^2 \rho^{2n-1}
\end{split}
\end{equation*}
for all sufficiently small $\rho$.  (Recall that $\delta$ is of order 
$\rho^{n-1}$ in this case.)

As a next step, let $t \in [\pi/2, \pi]$.  We apply the same reasoning and 
determine all $t$ with~\eqref{eqn:tinterval_plus}.
Let $s = \pi - t \in [0, \pi/2]$.  Then $\sin(s) = \sin(t)$, and 
$s$ satisfies~\eqref{eqn:tinterval_plus2}.
Further, $\abs{\re(a \rho^n e^{int})} = \abs{\re(a\rho^n e^{-isn})}$.  Note that 
switching the sign of $t$ in the previous case does not alter the proof, so we 
obtain again~\eqref{eqn:ming_reanneq0} and~\eqref{eqn:ming_rean0}.

For $t \in [\pi, 2 \pi]$ the sine is negative, so~\eqref{eqn:tinterval_plus} 
has to be replaced by
\begin{equation*}
\frac{2}{\rho^{n-1}} \sin(t) = - \abs{a} - \tau,
\quad - \delta \leq \tau \leq \delta,
\end{equation*}
with $\delta$ as before.
For $t \in [\pi, 3\pi/2]$ the substitution $s = t - \pi \in [0, \pi/2]$ brings 
us back to the first case $t \in [0, \pi/2]$, so that we 
obtain~\eqref{eqn:ming_reanneq0} and~\eqref{eqn:ming_rean0}.
For $t$ in $[3\pi/2, 2\pi]$ we substitute $s = 2 \pi - t \in [0, \pi/2]$ so 
that $\sin(t) = - \sin(s)$ and $s$ satisfies~\eqref{eqn:tinterval_plus}.  As in 
the second case $\re(a \rho^n e^{int}) = \re(a \rho^n e^{-isn})$, and we 
obtain~\eqref{eqn:ming_reanneq0} and~\eqref{eqn:ming_rean0}.
\end{proof}

The following lemma shows that the index of $f$ at the origin depends on the 
first two nonvanishing coefficients in the Taylor series of $h$.

\begin{lem} \label{lem:cutoff_reneq0}
Let
\begin{equation*}
f(z) = h(z) - \conj{z} = z + a_n z^n + \sum_{k=n+1}^\infty a_k z^k - \conj{z}
\end{equation*}
with $\re(a_n) \neq 0$.  Then $f$ and
\begin{equation*}
g(z) = z + a_n z^n - \conj{z}
\end{equation*}
have the same index at the origin: $\ind(f; 0) = \ind(g; 0)$.
\end{lem}

\begin{proof}
We aim to apply Rouch\'e's theorem on a sufficiently small circle around the 
origin.
Let $R > 0$ be smaller than the radius of convergence of the power series of $h$ 
at $0$, and so that no further zeros of $f$ and $g$ are contained in the 
closed $R$-disk.  In particular $M = \sum_{k=n+1}^\infty a_k R^{k-(n+1)} < 
\infty$.  Let $0 < \rho < R$.  We then have for $\abs{z} = \rho$
\begin{equation*}
\abs{f(z) - g(z)} = \abs{ \sum_{k=n+1}^\infty a_k z^k}
\leq \rho^{n+1} \sum_{k=n+1}^\infty \abs{a_k} \rho^{k-(n+1)}
= \rho^{n+1} M.
\end{equation*}
Together with the bound from Lemma~\ref{lem:ming_reanneq0} we obtain the strict 
inequality
\begin{equation*}
\abs{f(z) - g(z)} \leq \rho^{n+1} M < c \abs{\re(a_n)} \rho^n < \abs{g(z)} 
\leq \abs{g(z)} + \abs{f(z)}
\end{equation*}
for all sufficiently small $\rho$, so that $f$ and $g$ have the same winding 
along $\abs{z} = \rho$ by Rouch\'e's theorem, showing the assertion.
\end{proof}

The next lemma shows that scaling the coefficient $a_n$ does not change the 
index.

\begin{lem} \label{lem:rays}
Let $n \geq 2$ and let $a, b \in \C$ be nonzero with same argument.
Then $g(z) = a z^n + z - \conj{z}$ and
$\widetilde{g}(z) = b z^n + z - \conj{z}$ have the same index at the origin: 
$\ind(g; 0) = \ind(\widetilde{g}; 0)$.
\end{lem}

\begin{proof}
Let $\gamma(t) = \rho e^{i t}$ with $ t \in [0, 2 \pi]$.
We show that the curves $g \circ \gamma$ and $\widetilde{g} \circ \gamma$ are 
homotopic in $\C \backslash \{ 0 \}$ for all sufficiently small $\rho$.
Define
\begin{equation*}
H(t,s) = s g(\gamma(t)) + (1-s) \widetilde{g}(\gamma(t)), \quad (t,s) \in [0, 
2\pi] \times [0, 1].
\end{equation*}
Note that $H(0, s) = H(2\pi, s)$ for all $s$, i.e., $H(\cdot, s)$ is a closed 
curve for each $s$.
Writing $a = \abs{a} e^{i \varphi}$ and $b = \abs{b} e^{i \varphi}$ (they 
have the same argument by assumption) gives
\begin{equation*}
H(t,s) = (s \abs{a} + (1-s) \abs{b}) \rho^n e^{i (n t + \varphi)} + 2 i \rho 
\sin(t),
\end{equation*}
and we have to show that $H(t,s)$ is nonzero for all $(t,s)$ and for all 
sufficiently small $\rho$.  If $t$ is a multiple of $\pi$, we have
\begin{equation*}
\abs{H(k \pi, s)} = (s \abs{a} + (1-s) \abs{b}) \rho^n > 0, \quad k = 0, 1, 2.
\end{equation*}
For $t \neq 0, \pi, 2\pi$ the term $2 i \rho \sin(t)$ is nonzero and purely 
imaginary.  For $H(t,s)$ to become zero, the first term
$(s \abs{a} + (1-s) \abs{b}) \rho^n e^{i (n t + \varphi)}$ must thus also be 
imaginary, which can happen at at most finitely many values of $t$ (independent 
of $s$).  At each such point, making $\rho$ sufficiently small guarantees that 
$H(t,s) \neq 0$.  Thus $H(t,s) \neq 0$ for all $(t,s)$ and all sufficiently 
small $\rho$, which shows that the curves $g \circ \gamma$ and 
$\widetilde{g} \circ \gamma$ are homotopic in $\C \backslash \{ 0 \}$ and thus 
have the same winding.  Since this holds for all sufficiently small $\rho$, 
we obtain $\ind(g; 0) = \ind(\widetilde{g}; 0)$.
\end{proof}

The previous lemma shows that the index of $g(z) = a z^n + z - \conj{z}$ at the 
origin is the same for all $a$ on a ray starting in the origin.  Next we show 
that the index is also the same if we displace $a$ in its (open) half-plane.

\begin{lem} \label{lem:halfplanes}
Let $n \geq 2$ and $\re(a) \neq 0$.  Then the functions $g(z) = a z^n + z - 
\conj{z}$ and $\widetilde{g}(z) = \sign(\re(a)) z^n + z - \conj{z}$ have the 
same index at the origin: $\ind(g; 0) = \ind(\widetilde{g}; 0)$.
\end{lem}

\begin{proof}
Let $a \in \C$ with $\re(a) > 0$ and write $g_a(z) = a z^n + z - \conj{z}$ and 
$g_1(z) = z^n + z - \conj{z}$.
Let $\gamma(t) = \rho e^{i t}$ with $\rho > 0$.  We show that the closed curves 
$g_a \circ \gamma$ and $g_1 \circ \gamma$ are homotopic in $\C \backslash \{ 0 
\}$, provided that $\rho$ is sufficiently small.  Define
\begin{equation*}
    H(t,s) = s g_a(\gamma(t)) + (1-s) g_1(\gamma(t)), \quad (t,s) \in 
[0, 2\pi]\times [0,1],
\end{equation*}
which satisfies $H(0, s) = H(2 \pi, s)$ for all $s$, i.e., each $H(\cdot, s)$ 
is a closed curve.  Since
\begin{equation*}
H(t,s) = (1 + s (a-1)) (\gamma(t))^n + \gamma(t) - \conj{\gamma(t)},
\end{equation*}
Lemma~\ref{lem:ming_reanneq0} shows that for all $t$ and all sufficiently small 
$\rho$
\begin{equation*}
\abs{H(t,s)} \geq c \abs{\re (1 + s (a-1)) } \rho^n 
\geq c \min \{ 1, \re(a) \} \rho^n > 0,
\end{equation*}
so that $\abs{H(t,s)} > 0$ for all $(t,s)$ and sufficiently small $\rho$.
This shows that $g_a \circ \gamma$ and $g_1 \circ \gamma$ are homotopic in $\C 
\backslash \{ 0 \}$, so that $g_a$ and $g_1$ have the same winding along 
$\abs{z} = \rho$.  Since this holds for all sufficiently small $\rho$, their 
indices are the same.
Note that we needed that $a$ and $1$ are on the same side of the imaginary 
axis, so that $H(t,s)$ is guaranteed to be nonzero.
Similarly $g_a$ with $\re(a) < 0$ and $g_{-1}$ are homotopic in $\C 
\backslash \{ 0 \}$.
\end{proof}

The next lemma shows that to compute the index of $a z^n + z - \conj{z}$ at the
origin, we can reduce the power $n$ in steps of $2$, provided that $\re(a) \neq
0$.

\begin{lem} \label{lem:nreduction}
Let $n \geq 2$ and $\re(a) \neq 0$.  Then $g(z) = az^{n+2} + z - \conj{z}$ and
$\widetilde{g}(z) = a z^n + z - \conj{z}$ have the same index at the origin.
%
\end{lem}

\begin{proof}
We show that $g(z) = a z^{n+2} + z - \conj{z}$ and $\widetilde{g}(z) = a z^n + 
z - \conj{z}$ have the same winding on all sufficiently small circles around 
the origin using Rouch\'e's Theorem~\ref{thm:Rouche}.

We can assume that $a = \pm 1$ by Lemma~\ref{lem:halfplanes}.
Write $z = \rho e^{it}$ with $\rho > 0$ and $t \in [0, 2\pi]$.
To apply Rouch\'e's theorem, we wish to show the inequality
\begin{equation*}
\abs{\widetilde{g}(z)-g(z)}
= \abs{a \rho^n e^{int} - a \rho^{n+2} e^{i(n+2)t}}
< \abs{\widetilde{g}(z)}
= \abs{ a \rho^n e^{int} + 2i \rho \sin(t)}
\end{equation*}
for all $0 \leq t \leq 2 \pi$, or equivalently
\begin{equation}
\abs{\rho^{n-1} - \rho^{n+1} e^{i2t}}^2
< \abs{ a \rho^{n-1} e^{int} + 2i \sin(t)}^2, \quad 0 \leq t \leq 2 \pi,
\label{eqn:rouche_cond}
\end{equation}
for all sufficiently small $\rho > 0$.
The left and right hand sides in~\eqref{eqn:rouche_cond} are
\begin{equation*}
\begin{split}
\abs{\rho^{n-1} - \rho^{n+1} e^{i2t}}^2
&= \rho^{2n-2} - 2 \rho^{2n} \cos(2t) + \rho^{2n+2}, \\
\abs{a \rho^{n-1} e^{int} + 2i \sin(t)}^2
&= \rho^{2n-2} + 4 a \rho^{n-1} \sin(t) \sin(nt) + 4 \sin(t)^2,
\end{split}
\end{equation*}
respectively.  Thus~\eqref{eqn:rouche_cond} is equivalent to
\begin{equation}
F(t) \coloneq 4 \sin(t)^2 + 4 a \rho^{n-1} \sin(t) \sin(nt) + 2 \rho^{2n} 
\cos(2t) - \rho^{2n+2} > 0
\label{eqn:rouche_cond_explicit}
\end{equation}
for all $t \in [0, 2 \pi]$ and all sufficiently small $\rho > 0$.

Fix $0 < \delta < \frac{\pi}{4}$, so that $\cos(2\delta) > 0$ and 
$\sin(\delta) > 0$.  For $\abs{t} \leq \delta$ and $\abs{t-\pi} \leq \delta$ we 
compute
\begin{equation*}
F(t) = 4 \sin(t)^2 \left( 1 + a \rho^{n-1} \frac{\sin(nt)}{\sin(t)} \right)
+ \rho^{2n} ( 2 \cos(2t) - \rho^2 ).
\end{equation*}
Note that $\abs{a \sin(nt)/\sin(t)} \leq M < \infty$ on $[-\delta, \delta]$ 
and $[\pi - \delta, \pi + \delta]$, since $t = 0$ and $t = \pi$ are removable 
singularities.  Therefore,
\begin{equation*}
F(t) \geq 4 \sin(t)^2 ( 1 - M \rho^{n-1} ) + \rho^{2n} ( 2 \cos(2\delta) - 
\rho^2),
\end{equation*}
which is positive for all sufficiently small $\rho > 0$.
For $\delta \leq t \leq \pi - \delta$ and $\pi + \delta \leq t \leq 2 \pi - 
\delta$ we have $\sin(t)^2 \geq \sin(\delta)^2 > 0$, so that
$F(t) \geq 4 \sin(\delta)^2 + \bigO(\rho)$, which is positive for all 
sufficiently small $\rho > 0$.

This establishes~\eqref{eqn:rouche_cond_explicit} for all $t \in [0, 2\pi]$ 
and all sufficiently small $\rho > 0$, so that $\ind(g; 0) = 
\ind(\widetilde{g}; 0)$ by Rouch\'e's theorem.
\end{proof}

The next lemma completely characterizes the index of the harmonic polynomials
$a z^n + z - \conj{z}$ at the origin, for all $n \geq 2$ and all nonzero $a$.

\begin{lem} \label{lem:compute_index}
Let $g(z) = a z^n + z - \conj{z}$ with nonzero $a \in \C$ and $n \geq 2$.
We then have for even $n$
\begin{equation*}
\ind(g; 0) = \begin{cases} 0 & \re(a) \neq 0 \\ 1 & \re(a) = 0 \end{cases},
\end{equation*}
and for odd $n$
\begin{equation*}
\ind(g; 0) = \begin{cases} +1 & \re(a) \geq 0 \\ -1 & \re(a) < 0  \end{cases}.
\end{equation*}
\end{lem}

\begin{proof}
We treat the cases $\re(a) \neq 0$ and $\re(a) = 0$ separately.

\paragraph*{\bfseries Case $\bm{\re(a) \neq 0$}.}
We distinguish the cases of even and odd $n$.  First, let 
$n$ be even, so that we can assume $n = 2$ by Lemma~\ref{lem:nreduction}.
By Lemma~\ref{lem:halfplanes} we can even assume that $a \neq 0$ is real (and 
even $\pm 1$).

We compute the zeros of $g(z) = a z^2 + z - \conj{z}$.  Writing $z = x + i y$ 
with $x, y \in \R$ we find
\begin{equation*}
g(z) = a (x+iy)^2 + 2 i y = a (x^2 - y^2) + 2i (a x y + y).
\end{equation*}
Thus $g(z) = 0$ if and only if
\begin{equation*}
x^2 - y^2 = 0 \quad \text{and} \quad y (a x + 1) = 0.
\end{equation*}
The second equation is zero if $y = 0$ (thus $x = 0$) or $a x + 1 = 0$, i.e., 
$x = - 1/a$ and thus $y = \pm 1/a$.
This gives the three solutions $z_0 = 0$, $z_+ = (-1 + i)/a$ and $z_- = 
(-1-i)/a$, and we compute their indices.  Let $h(z) = a z^2 + z$.
Then $\abs{h'(z_{\pm})} = \abs{-1 \pm 2i} > 1$ shows that 
$z_\pm$ are sense-preserving zeros, so that $\ind(g; z_\pm) = 1$ by 
Proposition~\ref{prop:regular}.

On a sufficiently large circle $\Gamma$, the winding of $g$ is $2$ by 
Proposition~\ref{prop:global_winding}.  Applying the argument 
principle~\ref{thm:argument_principle} on $\Gamma$ shows
\begin{equation*}
2 = V(g; \Gamma) = \ind(g; z_+) + \ind(g; z_-) + \ind(g; 0) = 2 + \ind(g; 0),
\end{equation*}
i.e., $\ind(g; 0) = 0$.
This concludes the case of $\re(a) \neq 0$ and even $n$.


Next we consider the case $\re(a) \neq 0$ and $n$ odd.
By Lemma~\ref{lem:nreduction} we can assume that $n = 3$.
Note that the winding of $g(z) = a z^3 + z - \conj{z}$ around a sufficiently 
large circle is $3$.  Let $h(z) = a z^3 + z$.  
As before, let $z = x + i y$ with real $x$ and $y$, and compute the zeros of 
$g$ explicitly.

For real $a > 0$ we find the three zeros
\begin{equation*}
z_0 = 0, \quad z_{\pm} = \pm i \sqrt{2/a}.
\end{equation*}
Since $\abs{h'(z_{\pm})} = 5 > 1$, the zeros $z_+$ and $z_-$ are 
sense-preserving and have index $+1$.  The argument principle applied on a 
sufficiently large circle now shows that $\ind(g; 0) = +1$.
Lemma~\ref{lem:halfplanes} shows that $\ind(g; 0) = +1$ for all $a$ with 
$\re(a) > 0$.

For real $a < 0$ we find the five zeros
\begin{equation*}
z_0 = 0, \quad
z_{\pm} = \pm \sqrt{\sqrt{3}/(4 \abs{a})} + i \sqrt{1/(4\abs{a})}, \quad
\conj{z_{\pm}}.
\end{equation*}
A short computation gives $\abs{h'(z_{\pm})} = \abs{h'(\conj{z_{\pm}})} > 1$, 
so that these zeros have index $+1$ by Proposition~\ref{prop:regular}.  
Finally, the argument principle applied to a sufficiently large circle implies 
$\ind(g;0 ) = -1$.
Lemma~\ref{lem:halfplanes} shows that $\ind(g; 0) = -1$ for all $a$ with 
$\re(a) < 0$.  This concludes the case $\re(a) \neq 0$ and odd $n$.


\paragraph*{\bfseries Case $\bm{\re(a) = 0}$.} By Lemma~\ref{lem:rays} we can assume 
that $a = \pm i$ and we treat the case $a = i$ first.
We explicitly compute the zeros of $g(z) = i z^n + z - \conj{z}$.
Let $z = \rho e^{i \varphi}$ with $\rho > 0$ and $\varphi \in [0, 2\pi[$.
Then $g(z) = 0$ is equivalent to
\begin{equation*}
\rho^{n-1} e^{i n \varphi} + 2 \sin(\varphi) = 0,
\end{equation*}
and considering the real and imaginary parts separately we obtain the
pair of equations
\begin{align}
&\rho^{n-1} \cos(n \varphi) + 2 \sin(\varphi) = 0, \label{eqn:zer_ai_1} \\
&\rho^{n-1} \sin(n \varphi) = 0. \label{eqn:zer_ai_2}
\end{align}
Equation~\eqref{eqn:zer_ai_2} is equivalent to $n \varphi = k \pi$ 
with $k \in \Z$, giving the angles
\begin{equation*}
\varphi_k = k \frac{\pi}{n}, \quad k = 0, 1, 2, \ldots, 2n-1.
\end{equation*}
Inserting these in~\eqref{eqn:zer_ai_1} gives
\begin{equation}
\rho^{n-1} = (-1)^{k+1} 2 \sin \left(k \frac{\pi}{n} \right) \label{eqn:rhok}
\end{equation}
which must be positive.  In particular, $k = 0, n, 2n$ are not admissible.  For 
$k = 1, 2, \ldots, n-1$ the sine is positive, and thus $k$ must be odd, and for 
$k = n+1, n+2, \ldots, 2n-1$, the sine is negative and thus $k$ must be even.
Let us count precisely the number of admissible angles.  First, let $n$ be 
even.  Then $k = 1, 3, \ldots, n-1$, give $n/2$ solutions, and $k = n+2, n+4, 
\ldots, 2n-2$ give another $n/2 - 1$ solutions.
Second, let $n$ be odd.  Then $k = 1, 3, \ldots, n-2$ give $(n-1)/2$ solutions, 
and $k = n+1, n+3, \ldots, 2n-2$ are $(n-1)/2$ solutions.  

In either case we thus have $n-1$ zeros $z_k = \rho_k e^{i \varphi_k}$, 
where $\rho_k$ is given by~\eqref{eqn:rhok}, and we show that these are 
sense-preserving.  Let $h(z) = i z^n + z$.  We then have
\begin{equation*}
\begin{split}
\abs{h'(z_k)} &\geq \abs{i n z_k^{n-1} } - 1
= \abslr{ i n (-1)^{k+1} 2 \sin \left(k \frac{\pi}{n} \right) e^{i (n-1) 
\varphi_k} } - 1 \\
&= 2 n \abslr{ \sin \left( k \frac{\pi}{n} \right) } - 1
\geq 2 n \sin (\pi/n) - 1
\geq 4 - 1 > 1.
\end{split}
\end{equation*}
In the last estimate we used that $\pi \geq \sin(\pi x)/x \geq 2$ for
$x \in [0, 1/2]$, readily established by basic calculus.
Therefore these $n-1$ zeros are sense-preserving and have index $+1$.  Since 
the winding of $g$ on a sufficiently large circle is $n$, this implies
$\ind(g; 0) = +1$.

It remains to consider the case $a = -i$.
The function $g(z) = -i z^n + z - \conj{z}$ has a zero at the origin.  As in 
the case $a = i$ we explicitly compute all other zeros and show that they are 
regular.
Writing $z = \rho e^{i \varphi}$, we find that $g(z) = 0$ is equivalent to
\begin{equation*}
\rho^{n-1} e^{i n \varphi} - 2 \sin(\varphi) = 0,
\end{equation*}
which gives again the angles $\varphi_k = k \pi/n$, $k = 0, 1, \ldots, 2n-1$, 
and the corresponding radii
\begin{equation*}
\rho_k^{n-1} = (-1)^k 2 \sin \left( k \frac{\pi}{n} \right).
\end{equation*}
(Note the flipped sign compared to the case $a = i$.)
Now, for $k = 1, 2, \ldots, n-1$ we must have $k$ even, and for $k = n+1, 
\ldots, 2n-1$ we must have $k$ odd.
In each case we find again $n-1$ solutions.  The same computation as before 
shows $\abs{h'(z_k)} > 3$, so that the $n-1$ zeros are sense-preserving with 
index $+1$, implying $\ind(g; 0) = 1$ as before.
\end{proof}

With Lemma~\ref{lem:compute_index} we have completed the proof of 
Theorem~\ref{thm:index_sing_zero}.

\subsection{The examples revisited}
\label{sect:examples2}

In the previous Section~\ref{sect:examples} we had shown various
examples of functions $f$ having a singular zero $z_0$, and we
computed their indices.  The cumbersome computation entailed the
determination of the winding on a sufficiently large circle that
encloses \emph{all} exceptional points of $f$, and determining the
index of all such points except for $z_0$.  We then used the argument
principle to finally determine the index of $z_0$.

Using Theorem~\ref{thm:index_sing_zero} we are now able to compute these
indices directly.  For Example~\ref{ex:index+1}, where $f(z) = z /
(1 - z^2) - \conj{z}$ has a singular zero in $0$, we find
\begin{equation}
\label{eq:series_1}
\frac{z}{1 - z^2} = z + z^3 + \bigO(z^5),
\end{equation}
so that the first non-vanishing power after $z$ in the series
expansion is $n=3$, with corresponding coefficient $a_3 = 1$.  Hence,
by our classification, the index is $+1$.  From the series
expansion~\eqref{eq:series_1} one also finds that the index of the
isolated singular zero $0$ of $f(z) = -z / (1 - z^2)$ is $-1$
(cf.~Example~\ref{ex:index-1}).

In Example~\ref{ex:index0} we considered the function $f(z) = 2z^3 +
1/(8z) - \conj{z}$, and we found that $z_0 = i/2$ is one out of four
singular zeros having index $0$.  We develop the analytic part of $f$
in a power series around $z_0$, i.e.,
\begin{equation*}
    2z^3 + \frac{1}{8z} = \frac{i}{2}
        - \left(z - \frac{i}{2} \right)
        + 4i \left(z - \frac{i}{2} \right)^2
        + \bigO\left( \left(z - \frac{i}{2} \right)^4 \right).
\end{equation*}
In contrast to the previous example, the coefficients in this
expansion are not normalized as in Theorem~\ref{thm:index_sing_zero}, so
we resort to the general form of our characterization, given in
Theorem~\ref{thm:index_sing_zero_general}.  We have $a_1 = -1 =
e^{i\theta}$ with $\theta=\pi$, and $0 \neq a_2 = 4i = 4e^{i\varphi}$,
with $\varphi = \pi/2$, so that $\eta = \cos(\pi/2 - 3/2 \pi)=-1$ (see
definition~\eqref{eq:eta}).
Since $n=2$ is even and $\eta \neq 0$ we obtain $\ind(f; i/2) = 0$.

In the final Example~\ref{ex:exp} we considered the function $f(z) =
\exp(z) - 1 - \conj{z}$.  Lacking of tools to compute the index of the
singular zero $z_0 = 0$, we resorted to the phase portrait of $f$ (see
Figure~\ref{fig:sing_examples}), from which we read that the index
should be zero.  Developing $\exp(z) - 1$ in a series around the
origin, i.e.,
\begin{equation*}
    \exp(z) - 1 = z + \frac{z^2}{2} + \bigO(z^3),
\end{equation*}
we see that $n=2$ is the first non-vanishing power, and that the
corresponding coefficient is $a_2 = 1/2$.  From the classification in
Theorem~\ref{thm:index_sing_zero} it follows that $\ind(f; 0) = 0$.

\section{Conclusions and future work}
\label{sect:imag_diff}
\label{sect:conclusion}

In this work we developed a technique to determine the index of
\emph{singular} zeros of $f(z) = h(z) - \conj{z}$.  In summary, this
index depends only on the first two non-vanishing coefficents of the
power series of $h$ at the zero.  Our classification is \emph{almost}
always applicable.  As discussed in the
Remarks~\ref{rem:comment1}--\ref{rem:comment3},  we had to exclude one
particular coefficient configuration from our classification.  The
reason is that in this case the index $\ind(f; z_0)$ is \emph{not}
entirely defined by these first two coefficients.

\begin{figure}[t]
\begin{center}
    \includegraphics[width=0.49\textwidth]{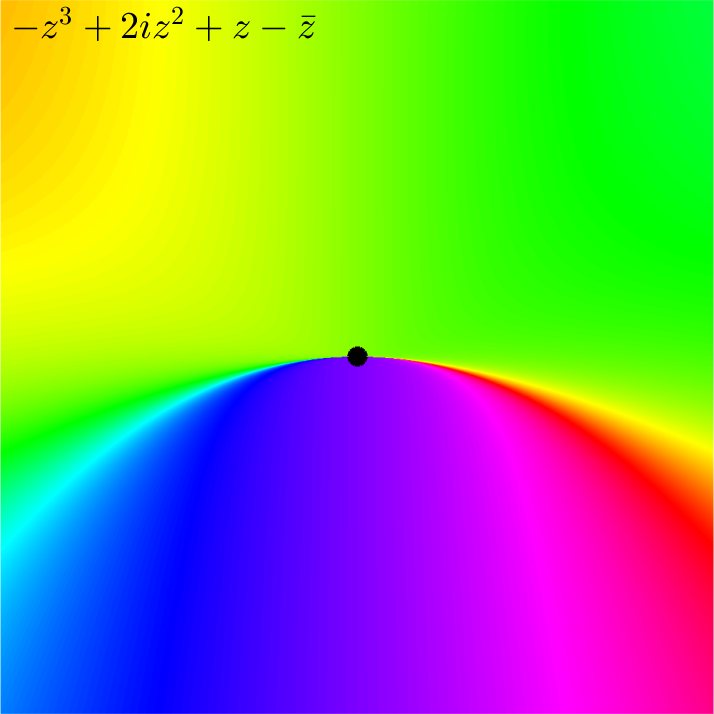}
    \hfill
    \includegraphics[width=0.49\textwidth]{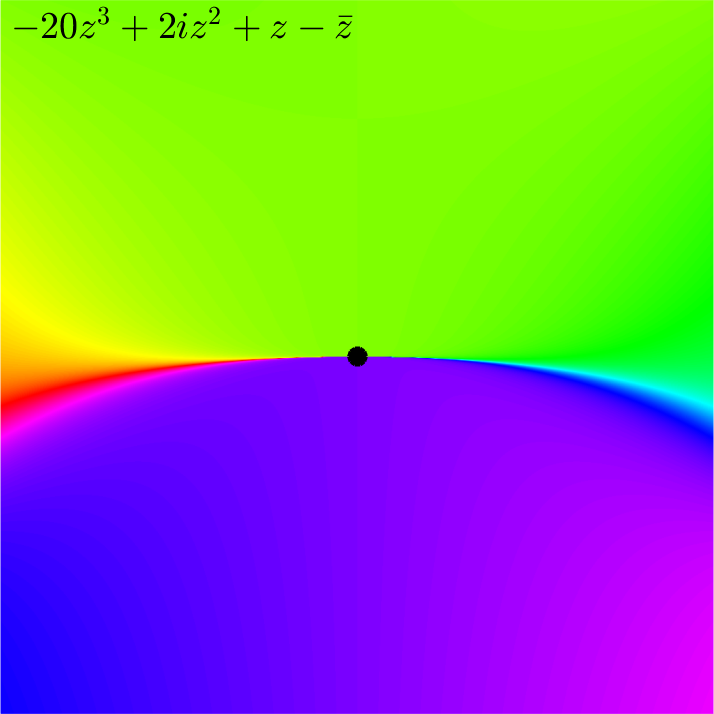}
    \caption{An example for the excluded case in
    Theorem~\ref{thm:index_sing_zero}; see
    Section~\ref{sect:conclusion}.\label{fig:difficult}}
\end{center}
\end{figure}

In order to illustrate this behaviour, consider the functions
\begin{equation*} f_1(z) = -z^3 + iz^2 + z - \conj{z} \quad \text{and}
\quad f_2(z) = -20z^3 + iz^2 + z - \conj{z}, \end{equation*} which are
shown in Figure~\ref{fig:difficult}.  Both functions have a singular
zero at the origin, and their Taylor series up to order two is
identical.  Since the coefficient $a_2 = 2i$ is purely imaginary, our
classification in Theorem~\ref{thm:index_sing_zero} does not apply.
Indeed their indices at $0$ are $+1$ and $-1$, implying that the first
two coefficients are \emph{not} sufficient to determine the index.
Our preliminary investigation of this case, i.e., where the second
non-vanishing coefficient is purely imaginary, has led us to the
belief that this situation is much more irregular, and that a
thorough investigation is to be carried out in future work.

Another interesting extension of our results would be the
consideration of a general anti-analytic part, i.e., general harmonic
mappings $f = h + \conj{g}$.  Since the index is a local property in
this case as well, one could hope that a similarly flavoured
characterization can be obtained in this general setting.

\opt{preprint}{
\bibliographystyle{siamplain}
}
\opt{tandf}{
\bibliographystyle{gCOV}
}
\bibliography{singular}

\end{document}